\documentclass[reqno]{amsart}
\usepackage{amsmath}
\usepackage{amssymb}
\usepackage{amscd}
\usepackage{mathptmx}
\usepackage{amssymb}
\usepackage{amsfonts}
\usepackage{amsmath}
\usepackage{graphicx}

\begin{document}
\title {Extended Armendariz Rings}

\author{Nazim Agayev}
\address{Nazim Agayev,  Qafqaz University, Department of Pedagogy,
Baku, Azerbaijan} \email{nazimagayev@qafqaz.edu.az}

\author{Abdullah Harmanci}
\address{Abdullah Harmanci, Hacettepe University, Department of Mathematics, Ankara Turkey}
\email{harmanci@hacettepe.edu.tr}

\author{Sait Halicioglu}
\address{Sait Hal\i c\i oglu,  Department of Mathematics, Ankara University, 06100 Ankara, Turkey}
\email{halici@ankara.edu.tr}

\date{}
\maketitle
\newtheorem {thm}{Theorem}[section]
\newtheorem{lem}[thm]{Lemma}
\newtheorem{prop}[thm]{Proposition}
\newtheorem{cor}[thm]{Corollary}
\newtheorem{df}[thm]{Definition}
\newtheorem{nota}{Notation}
\newtheorem{note}[thm]{Remark}
\newtheorem{ex}[thm]{Example}
\newtheorem{exs}[thm]{Examples}
\newtheorem{rmk}[thm]{Remark}
\newtheorem{quo}[thm]{Question}

\begin{abstract} In this note we introduce  central linear
Armendariz rings as a generalization of  Armendariz rings  and
investigate their properties.\\


\end{abstract}

\section{ Introduction } Throughout this paper $R$ denotes an associative ring with
identity. Rege and Chhawchharia \cite{RC},  introduce the notion
of an Armendariz ring. The ring $R$ is called {\it Armendariz~}
if for any $f(x)=\sum_{i=0}^n a_ix^i$,  $g(x)=\sum_{j=0}^s b_jx^j
\in R[x]$, $f(x)g(x)=0$ implies $a_ib_j=0$ for all $i$ and  $j$.
The name of the ring was given due to Armendariz who proved that
reduced rings (i.e. rings without nonzero nilpotent elements)
satisfied this condition \cite{AE}.

Number of papers have been written on the Armendariz rings (see,
e.g.  \cite{AC1}, \cite{KLO}). So far,  Armendariz rings are
generalized in different ways (see namely, \cite{HKKwak},
\cite{LZh}). In particular, Lee and Wong \cite{LW} introduced {\it
weak Armendariz rings} (i.e. if the product of two linear
polynomials in $R[X]$ is $ 0$, then each product of their
coefficients is 0), Liu and Zhao \cite{LZh} introduce also {\it
weak Armendariz rings} ( if the product of two polynomials in
$R[X]$ is $ 0$, then each product of their coefficients is
nilpotent) as another generalization of Armendariz rings. To get
rid of confusion, we call the rings {\it linear Armendariz} which
satisfy Lee and Wong condition.
 A ring $R$ is called {\it central linear Armendariz},  if the product of two linear polynomials in $R[X]$ is $ 0$, then each product of their
coefficients is central.  Clearly, Armendariz rings are  linear
Armendariz and linear Armendariz rings are central linear
Armendariz. In case $R$ is reduced ring every weak Armendariz ring
is central linear Armendariz. We supply some examples to show that
the converses of these statements need not be  true in general. We
prove that the class of central linear Armendariz rings ýÜülies
strictly between classes of linear Armendariz rings and abelian
rings.    For a ring $R$, it is shown that  the polynomial ring
$R[x]$ is central linear Armendariz if and only if the Laurent
polynomial ring $R[x, x^{-1}]$ is central linear Armendariz. Among
others we also show that  $R$ is reduced ring if and only if the
matrix ring $T_{n}^{k}(R)$ is Armendariz ring if and only if   the
matrix ring $T_{n}^{n-2}(R)$ is central  linear Armendariz ring,
for a natural number $n \geq 3$ and  $k=[n/2]$.  And for an ideal
$I$ of $R$, if $R/I$ central linear Armendariz and $I$ is reduced,
then $R$ is central linear Armendariz.

We also introduce central reduced rings as a generalization of
reduced  rings. The ring $R$ is called {\it central reduced} if
every nilpotent is central. We prove that if $R$ is central
reduced ring,  then $R$ is  central linear Armendariz, and if $R$
is central reduced ring, then the trivial extension $T(R,R)$ is
central linear Armendariz.  Moreover, it is proven that if $R$ is
a semiprime ring, then $R$ is central reduced ring   if and only
if   $R[x]/(x^{n})$ is central linear  Armendariz,  where $n\geq
2~$ is a natural number and $(x^n)$ is the ideal generated by
$x^n$.

 \indent\indent   We write $R[x],R[[x]],R[x,x^{-1}]$ and $R[[x,x^{-1}]]$ for the polynomial ring, the power series
ring, the Laurent polynomial ring and the Laurent power series
ring over $R$, respectively.

\section{Central Linear Armendariz Rings}
In this section central linear Armendariz rings are introduced as
a generalization of linear Armendariz rings. We prove that some
results of linear Armendariz rings can be extended to  central
linear Armendariz rings for this general settings. Clearly, every
Armendariz ring is linear Armendariz. However, linear  Armendariz
rings are not necessarily Armendariz in general (see \cite[Example
3.2 ]{LW}).

We now give a possible generalization of linear Armendariz rings.

\begin{df} The ring $R$ is called  {\it central linear Armendariz} if the product
of two linear polynomials in $R[X]$ is $ 0$, then each product of
their coefficients is central.
\end{df}

Note that all commutative rings, reduced rings, Armendariz rings
and linear Armendariz rings  are central linear Armendariz. It is
clear that subrings of central linear Armendariz rings are central
linear Armendariz.

Recall that $R$ is said to be {\it abelian} if idempotent elements
of $R$ are central.

\begin{lem}\label{lem1} If the ring $R$ is central linear Armendariz, then  $R$ is abelian.
\end{lem}
\begin{proof} Let $e$ be any idempotent in $R$, consider $f(x)=e-er(1-e)x, g(x)=(1-e)+er(1-e)x \in R[x]$
 for any $r\in R$. Then $f(x)g(x)=0$. By hypothesis, in particular $er(1-e)$ is central. Therefore
$er(1-e)=0$. Hence $er=ere$ for all $r\in R$. Similarly we
consider $h(x)=(1-e)-(1-e)rex$ and $t(x)=e+(1-e)rex$ in $R[x]$ for
any $r\in R$. Then $h(x)t(x)=0$. As before $(1-e)re=0$ and
$ere=re$ for all $r\in R$. It follows that $e$ is central element
of $R$, that is, $R$ is abelian.
\end{proof}

\begin{ex}   Let $R$ be any ring. For any integer $n\geq 2$, consider the ring $R^{n\times n}$ of $n\times n$ matrices  and the
ring $T_n(R)$ of $n\times n$ upper triangular matrices  over $R$.
The rings $R^{n\times n}$ and $T_n(R)$ contain non-central
idempotents. Therefore they are not abelian. By Lemma \ref{lem1}
these rings are not central linear Armendariz.
\end{ex}

Recall that a ring $R$ is {\it semicommutative}, if for any $a,b
\in R$, $ab=0$ implies $aRb=0$.

\begin{thm}\label{von} Let $R$ be a von Neumann regular ring $R$. Then the following are equivalent:\\
$(1)$ $R$ is Armendariz.\\
$(2)$ $R$ is reduced.\\
$(3)$ $R$ is central linear Armendariz.\\
$(4)$ $R$ is linear Armendariz.\\
$(5)$ $R$ is semicommutative.
\end{thm}
\begin{proof} By Lemma \ref{lem1} and \cite[Lemma 3.1, Theorem 3.2]{Go},  we have $(3)\Rightarrow(2)$.
$(2)\Rightarrow(5)$ Clear.  $(5)\Rightarrow(2)$ Let $a^2=0$ for
$a\in R$. By $(5)$, $aRa=0$. So $(aR)^2=0$. Assume  $aR\neq 0$. By
hypothesis, $aR$ contains a non-zero idempotent. This is a
contradiction. Hence $a=0$. The rest is clear from \cite[Theorem
6]{AC1}.
 \end{proof}

We now give a condition for a ring to be central linear Armendariz
relating to central idempotents.

\begin{lem}\label{iki} Let $R$ be a ring and $e$ an idempotent of $R$. If $e$ is a central idempotent of $R$, then the
following are equivalent: \\ {\rm (1)} $R$ is central linear
Armendariz.\\ {\rm (2)} $eR$ and $(1-e)R$ are central linear
Armendariz.
\end{lem}
\begin{proof} (1) $\Rightarrow$ (2) Since the subrings of central linear Armendariz rings are central linear Armendariz, $(2)$ holds. \\
(2) $\Rightarrow $ (1) Let $f(x)=a_0 + a_1 x$,  $g(x)=b_0+b_1 x$
be non zero polynomials in $R[x]$. Assume that $f(x)g(x)=0$. Let
$f_1=ef(x)$, $f_2=(1-e)f(x)$, $g_1=eg(x)$, $g_2=(1-e)g(x)$. Then
$f_1(x)g_1(x)=0$ in $(eR)[x]$ and $f_2(x)g_2(x)=0$ in
$((1-e)R)[x]$. By (2) $ea_ieb_j$ is central in $eR$ and
$(1-e)a_i(1-e)b_j$ is central in $(1-e)R$ for all $0\leq i\leq 1$,
$0\leq j\leq 1$.   Since $e$ and $1-e$ central in $R$, $R=
eR\oplus (1-e)R$ and so $a_ib_j=ea_ib_j + (1-e)a_ib_j$ is central
in $R$ for all $0\leq i\leq 1$, $0\leq j\leq 1$. Then $R$ is
central linear Armendariz.
\end{proof}

Clearly, any linear Armendariz ring is central linear Armendariz.
We now prove that the converse is true if the ring is right
$p.p.-$ring.

\begin{thm}\label{thm2}  If the ring $R$ is linear Armendariz, then  $R$ is central linear Armendariz. The converse holds if $R$ is right  $p.p.-$ring.
\end{thm}

\begin{proof}  Suppose $R$ is central linear Armendariz and right $p.p.-$ring.  Let
$f(x)=a_0 + a_1 x$,  $g(x)=b_0+b_1 x ~\in R[x]$. Assume
$f(x)g(x)=0$ Then we have:
\[
\begin{array}{lll}
 a_0b_0&=0&\hspace{0.8in}(1)\\
a_0b_1+a_1b_0&=0&\hspace{0.8in}(2)\\
a_1b_1&=0&\hspace{0.8in}(3)
\end{array}
\]
By hypothesis there exist  idempotents $e_i\in R$ such that
$r(a_i)=e_iR$ for all $i$. So  $b_0=e_0b_0$ and $a_0e_0=0$.
Multiply (2) from the right by $e_0$,  by Lemma \ref{lem1}, $R$ is
abelian and we have
$0=a_0b_1e_0+a_1b_0e_0=a_0e_0b_1+a_1b_0e_0=a_1b_0$. So $a_0b_1=0$.
Hence $R$ is linear Armendariz. This completes the proof.
\end{proof}

Let $R$ be a ring and let $M$ be an $(R, R)$-bimodule. The {\it
trivial extension} of $R$ by $M$ is defined to be the ring $T(R,M)
= R \oplus M$ with the usual addition and the multiplication
$(r_1,m_1)(r_2,m_2) = (r_1r_2,r_1m_2 + m_1r_2)$.

Example \ref{ters} shows that the assumption "right p.p.-ring" in
Theorem \ref{thm2} is not superfluous.

\begin{ex}\label{ters} There exists a central linear Armendariz ring which is neither right p.p.-ring nor linear Armendariz ring.
\end{ex}
\begin{proof} Let $n$ be an integer with $n \geq 2$. Consider the ring $R=T({\bf Z}_{2^n}, {\bf Z}_{2^n})$. If  $a=2^{n-1}$ and
$ f(x) =\left [
 \begin{array}{cc}
\bar{a} & \bar{0}\\
 \bar{0} & \bar{a} \end{array}\right] + \left [
 \begin{array}{cc}
\bar{a} & \bar{1}\\
 \bar{0} & \bar{a} \end{array}\right]x \in R[x], $
then $(f(x))^2=0$. Because $\left [
 \begin{array}{cc}
\bar{a} & \bar{0}\\
 \bar{0} & \bar{a} \end{array}\right]\left [
 \begin{array}{cc}
\bar{a} & \bar{1}\\
 \bar{0} & \bar{a} \end{array}\right] \neq 0$,  $R$ is not a  linear Armendariz ring. Since $R$ is commutative,  it is central linear Armendariz ring.
Moreover, since the principal ideal  $I=\left [
\begin{array}{cc}
0 & {\bf Z}_{2^n}\\
0 & 0 \end{array}\right]= \left [
\begin{array}{cc}
0 & 1\\
0 & 0 \end{array}\right]R$ is not projective, $R$ is not right
p.p.-ring.
\end{proof}

Now we will introduce a notation for some subrings of $T_n(R)$.
Let $k$ be a natural number smaller than $n$. Say \\ \\ \indent
$T_{n}^{k}(R)=\left
\{\displaystyle{\sum_{i=j}^{n}}\displaystyle{\sum_{j=1}^{k}}a_{j}e_{(i-j+1)i}+\sum_{i=j}^{n-k}\sum_{j=1}^{n-k}r_{ij}e_{j(k+i)}:
a_{j}, r_{ij}\in R\right \}$

\noindent where $e_{ij}$' s are matrix units. Elements of
$T_{n}^{k}(R)$
are in the form \\

\noindent
$$\left[\begin{array}{cccccccc}  x_{1} & x_{2} & ... & x_{k} &  a_{1(k+1)} &
a_{1(k+2)} & ...
& a_{1n}  \\
0 & x_{1} & ... & x_{k-1} & x_{k} &  a_{2(k+2)} &... & a_{2n}\\0&
0 & x_{1}& ... &  & & &
a_{3n}  \\
& & & ... &   &   &   & \\   &   &   &   &   &   &  & x_{1}
\end{array}\right]$$
  where $x_{i}, a_{js}\in R, $ $1\leq i\leq k, $ $1\leq j\leq n-k$
and  $k+1\leq s\leq n$.

For a reduced ring $R$, our aim is to investigate necessary and
sufficent conditions for $S=T_{n}^{k}(R)$ to be central linear
Armendariz. In \cite{LZ}, Lee  and Zhou prove that, if $R$ is
reduced ring, then $S$ is Armendariz ring for $k=[n/2]$. Hence $S$
is linear Armendariz and so  $S$ is central linear Armendariz. In
the following,
 we show that the converse of
this theorem is also true. Moreover,  it is proven that $R$ is
reduced ring if and only if  $T_{n}^{k}(R)$ is  Armendariz ring if
and only if $T_{n}^{n-2}(R)$ is  central linear Armendariz ring.
In this direction, we need the following lemma:

\begin{lem}\label{lemred} Suppose that  there exist $a,b\in R$ such
that $a^{2}=b^{2}=0$ and $ab=ba$ is not central. Then $R$ is not a
central linear Armendariz ring.
\end{lem}
\begin{proof} $(a+bx)(a-bx)=0$ in $R[x]$, but $ab$ is not central.
So, $R$ is not a central linear Armendariz ring.
\end{proof}

\begin{thm}\label{thmred} Let  $n\geq 3$ be a natural number. Then $R$ is reduced ring if and only if
$T_{n}^{k}(R)$ is central linear Armendariz ring, where $1\leq
k\leq n-2$.
\end{thm}
\begin{proof} Let $R$ be a reduced ring. In \cite{LZ}, it is shown that
 $T_{n}^{k}(R)$ is Armendariz ring and so it is central linear Armendariz. Conversely,
suppose that $R$ is not a reduced ring. Choose a nonzero element
$a\in R$ with square zero. Then for elements
$A=a(e_{11}+e_{22}+...+e_{nn}),
B=e_{1(k+1)}+e_{1(k+2)}+...+e_{1n}$ in $T_{n}^{k}(R)$,
$A^{2}=B^{2}=0$ and $AB=BA$ is not central, since
$(AB)(e_{1(n-k)}+e_{2(n-k+1)}+...+e_{k(n-1)}+e_{(k+1)n})=ae_{1n}\neq
0$. Therefore, from Lemma \ref{lemred}, $T_{n}^{k}(R)$ is not
central linear Armendariz ring. This completes the proof.
\end{proof}

\begin{thm}\label{thmdort}Let $R$ be a ring,  $n \geq 3$ be a natural number and  $k=[n/2]$.  Then the following are equivalent:\\
$(1)$ $R$ is reduced ring.\\
$(2)$ $T_{n}^{k}(R)$ is Armendariz ring.\\
$(3)$ $T_{n}^{n-2}(R)$ is central linear Armendariz ring.
\end{thm}
\begin{proof} $(1) \Rightarrow (2)$ See \cite{LZ}.\\
$(2) \Rightarrow (3)$ Since subrings of Armendariz rings are Armendariz, the rest is clear.\\
$(3) \Rightarrow (1)$ It follows from Theorem \ref{thmred}.
\end{proof}

Note that the homomorphic image of a central linear Armendariz
ring need not be central linear Armendariz. If $R$ is commutative
and Gaussian ring, by \cite[Theorem 8]{AC1}  every homomorphic
image of $R$ is Armendariz and so
 it is central linear Armendariz.

In \cite{HLS}, it was shown that for a ring $R$, if $I$ is a
reduced ideal of $R$ such that $R/I$ is Armendariz, then $R$ is
Armendariz. For central linear Armendariz rings we have the
similar result.

\begin{thm}\label{Ired}Let $R/I$ be central linear Armendariz and  I be reduced.
 Then R is central linear Armendariz.

\end{thm}

\begin{proof}
Let $a,b \in R$. If $ab=0$, then $(bIa)^{2}=0$. Since
$bIa\subseteq I $ and $I$ is reduced, $bIa=0$. Also,
$(aIb)^{3}\subseteq (aIb)(I)(aIb)=0$. Therefore $aIb=0$. Assume
$f(x)=a_{0}+a_{1}x, g(x)=b_{0}+b_{1}x\in R[x] $ and $f(x)g(x)=0$.
Then

\[
\begin{array}{lll}
 a_0b_0&=0&\hspace{0.8in}(1)\\
a_0b_1+a_1b_0&=0&\hspace{0.8in}(2)\\
a_1b_1&=0&\hspace{0.8in}(3)\\
\end{array}
\]
\\ We first show that for any  $a_{i}b_{j}$,  $a_{i}Ib_{j}=b_{j}Ia_{i}=0$. Multiply $(2)$
from the right by $Ib_{0}$, we have $a_{1}b_{0}Ib_{0}=0$, since
$a_{0}b_{1}Ib_{0}=0$. Then $(b_{0}Ia_{1})^{3}\subseteq
b_{0}I(a_{1}b_{0}Ia_{1}b_{0})Ia_{1}=0$. Hence $b_{0}Ia_{1}=0$.
This implies $a_{1}Ib_{0}=0$. Multiply $(2)$ from the  left by
$a_{0}I$, we have  $a_{0}Ia_0b_1+a_{0}Ia_1b_0=0$ and so
$a_{0}Ia_0b_1=0$. Thus $(b_{1}Ia_{0})^{3}=0$ and $b_{1}Ia_{0}=0$.
Therefore $a_{0}Ib_{1}=0$.
 Since
$R/I$ is central Armendariz, it follows that
$\overline{a_{i}}\overline{b_{j}}$ is central in $R/I$. So
$a_{i}b_{j}r-ra_{i}b_{j}\in I$ for any $r\in R$. Now from above
results, it can be easily seen that
$(a_{i}b_{j}r-ra_{i}b_{j})I(a_{i}b_{j}r-ra_{i}b_{j})=0$. Then
$a_{i}b_{j}r=ra_{i}b_{j}$ for all $r\in R$. Hence $a_{i}b_{j}$ is
central for all $i$ and $j$. This completes the proof.

\end{proof}

Let $S$ denote a multiplicatively closed subset of $R$ consisting
of central regular elements. Let $S^{-1}R$ be the localization of
$R$ at $S$. Then we have:
\begin{prop}\label{ilkp} $R$ is central linear Armendariz if and only if $S^{-1}R$
is central linear Armendariz.
\end{prop}
\begin{proof} Suppose that $R$ is a central linear  Armendariz ring. Let
$f(x)=\displaystyle{\sum_{i=0}^{1}}(a_{i}/s_i)x^{i}$,
$g(x)=\displaystyle{\sum_{j=0}^{1}}(b_{j}/t_j)x^{j} ~\in
(S^{-1}R)[x]$ and $f(x) g(x)=0$. Then we may find $u$, $v$, $c_i$
and $d_j$ in $S$ such that
$uf(x)=\displaystyle{\sum_{i=0}^{1}}a_{i}c_ix^{i}\in R[x]$,
$vg(x)=\displaystyle{\sum_{i=0}^{1}}b_{j}d_jx^{j}\in R[x]$ and
$(uf(x))(vg(x))=0$. By supposition $(a_{i}c_i)(b_{j}d_j)$ are
central in $R$ for all $i$ and $j$. Since $c_i$ and $d_j$ are
regular central elements of $R$, $a_ib_j$ are central in $R$ for
all $i$ and $j$. It follows that $(a_i/s_i)(b_j/t_j)$ are central
for all $i$ and $j$.
\\Conversely, assume that $S^{-1}R$ is a central linear Armendariz ring. Let
$f(x)=\displaystyle{\sum_{i=0}^{1}}a_{i}x^{i},
g(x)=\displaystyle{\sum_{j=0}^{1}}b_{j}x^{j} ~\in R[x]$.  Assume
$f(x)g(x)=0$. Then
$f(x)/1=\displaystyle{\sum_{i=0}^{1}}(a_{i}/1)x^{i},
g(x)=\displaystyle{\sum_{j=0}^{1}}(b_{j}/1)x^{j} ~\in S^{-1}R[x]$
and $(f(x)/1)(g(x)/1)=0$ in $S^{-1}R$. By assumption
$(a_i/1)(b_j/1)$ is central in $S^{-1}R$. Hence, for all $i$ and
$j$, $a_ib_j$ is central in $R$.
\end{proof}

\begin{cor}\label{cor1} For any ring $R$, the polynomial ring $R[x]$ is
central linear Armendariz if and only if the Laurent polynomial
ring $R[x, x^{-1}]$ is central linear Armendariz.
\end{cor}
\begin{proof} Let $S=\{1,x,x^2,x^3,x^4,...\}$. Then $S$ is a
multiplicatively closed subset of $R[x]$ consisting of central
regular elements. Then the proof follows from Proposition
\ref{ilkp}.
\end{proof}

We now define {\it central reduced rings} as a generalization of
reduced rings.

\begin{df} The ring $R$ is called  {\it central reduced ring} if every nilpotent element is central.
\end{df}

\begin{ex} All commutative rings,  all reduced rings and all strongly regular rings are central reduced.
\end{ex}

One may suspect that central reduced rings are reduced. But the
following example erases the possibility.

\begin{ex}
Let $S$ be a  commutative ring and $R=S[x]/(x^2)$. Then $R$ is
commutative ring and so  $R$ is central reduced.
 If $a=x+(x^2) \in R$, then $a^2=0$. Therefore $R$ is not  a reduced ring. \end{ex}

It is well known that if the ring $R$ is reduced, then $R$ is
linear Armendariz. In our case, we have the following:

\begin{thm}If $R$ is central reduced ring,  then $R$ is  central linear Armendariz.
\end{thm}

\begin{proof}Let $f(x)=a_{0}+a_{1}x$, $g(x)=b_{0}+b_{1}x \in R[x]$.  Assume
$f(x)g(x)=0$. Then we have :

\[
\begin{array}{lll}
 a_0b_0&=0&\hspace{0.8in}(1)\\
a_0b_1+a_1b_0&=0&\hspace{0.8in}(2)\\
a_1b_1&=0&\hspace{0.8in}(3)
\end{array}
\]

\noindent Since $(b_{0}a_{0})^2=0$ and $(b_{1}a_{1})^2=0$,
$b_{0}a_{0}, b_{1}a_{1}\in C(R)$, where $C(R)$ is the center of
$R$. Multiply (2) from the right by $a_0$,  we have
$a_{0}b_{1}a_{0}+a_{1}b_{0}a_{0}=0$. Thus
 $a_{0}b_{1}a_{0}+b_{0}a_{0}a_{1}=0$. Multiply last equation from the left by $a_{0}$, we have  ${a_{0}}^{2}b_{1}a_{0}=0$ and so  $(a_{0}b_{1}a_{0})^{2}=0$, that is,
$a_{0}b_{1}a_{0}\in C(R)$.  Hence $(a_{0}b_{1})^{3}=0$ and so
$a_{0}b_{1}\in C(R)$. Similarly it can be shown that
$a_{1}b_{0}\in C(R)$.
\end{proof}

Note that if $R$ is reduced ring, by \cite[Proposition 2.5]{RC}
trivial extension $T(R,R)$ is Armendariz and so it is linear
Armendariz. For central reduced rings,  we have

\begin{lem}\label{trr} If $R$ is central reduced ring, then the trivial extension $T(R,R)$ is central linear  Armendariz.
The converse holds if $R$ is semiprime.
\end{lem}

\begin{proof} Let $f(x)=\left [
\begin{array}{cc}
a_0 & b_0\\
0 & a_0 \end{array}\right]+ \left [
\begin{array}{cc}
a_1 & b_1\\
0 & a_1 \end{array}\right]x =\left [
\begin{array}{cc}
f_{1}(x) & f_{2}(x)\\
0 & f_{1}(x) \end{array}\right]$,\linebreak $g(x)=\left [
\begin{array}{cc}
c_0 & d_0\\
0 & c_0 \end{array}\right]+ \left [
\begin{array}{cc}
c_1 & d_1\\
0 & c_1 \end{array}\right]x=\left [
\begin{array}{cc}
g_{1}(x) & g_{2}(x)\\
0 & g_{1}(x) \end{array}\right] \in T(R,R)[x]$. If  $f(x)g(x)=0$,
then we have
 $$f(x)g(x)=\left [
\begin{array}{cc}
 f_{1}(x)g_{1}(x)& f_{1}(x)g_{2}(x)+f_{2}(x)g_{1}(x)\\
0 & f_{1}(x)g_{1}(x)
\end{array}\right]=0.$$  Hence $f_{1}(x)g_{1}(x)=0$, $f_{1}(x)g_{2}(x)+f_{2}(x)g_{1}(x)=0.$ In this case, we have

\[
\begin{array}{lll}
 a_0c_0&=0&\hspace{0.8in}(1)\\
a_0c_1+a_1c_0&=0&\hspace{0.8in}(2)\\
a_1c_1&=0&\hspace{0.8in}(3)\\
\end{array}
\]

\noindent From $(1)$ and $(3)$, $a_0c_0, a_1c_1\in C(R)$ and so
$c_{0}a_{0}, c_{1}a_{1}\in C(R)$. Multiply (2) from the right by
$a_0$,  we have  $a_{0}c_{1}a_{0}+a_{1}c_{0}a_{0}=0$. Thus
 $a_{0}c_{1}a_{0}+c_{0}a_{0}a_{1}=0$, so  ${a_{0}}^{2}c_{1}a_{0}=0$ and so  $(a_{0}c_{1}a_{0})^{2}=0$, that is,
$a_{0}c_{1}a_{0}\in C(R)$.  Hence $(a_{0}c_{1})^{3}=0$ and so
$a_{0}c_{1}\in C(R)$. Similarly it can be shown that
$a_{1}c_{0}\in C(R)$.\\
Conversely, suppose $R$ is semiprime and $S=T(R,R)$ is central linear  Armendariz. Let  $a^n=0$ with $a \in R$. Consider \\

\noindent $ f(x) =\left [
 \begin{array}{ll}
a^{n-1} & 0\\
 0 & a^{n-1} \end{array}\right] + \left [
 \begin{array}{ll}
a^{n-1} & 1\\
 0 & a^{n-1} \end{array}\right]x, \\
g(x) =\left [
 \begin{array}{ll}
a^{n-1} & 0\\
 0 & a^{n-1} \end{array}\right] + \left [
 \begin{array}{ll}
a^{n-1} & -1\\
 0 & a^{n-1} \end{array}\right]x   \in S[x]  $. Then $f(x)g(x)=0$. Hence $\left [
 \begin{array}{ll}
0 & a^{n-1} \\
 0 & 0  \end{array}\right] \in C(S)$ and so $a^{n-1} \in C(R)$. Therefore $(a^{n-1}R)^2=0$ implies $a^{n-1}=0$. Continuing in this way, we have $a=0$.
\end{proof}

In \cite[Theorem 5]{AC1},  Anderson and Camillo proved that for a
ring $R$ and \linebreak $n \geq 2$ a natural number,
$T_{n}^{n-1}(R)$ is Armendariz if and only if R is reduced. Lee
and Wong \cite[Theorem 3.1]{LW} also proved that $T_{n}^{n-1}(R)$
is linear Armendariz if and only if R is reduced.  For central
linear Armendariz rings, we have the following.

\begin{thm}\label{thm6} Let $R$ be a semiprime ring and $n\geq 2$ a natural number.
$R$ is central reduced ring if and only if  $T_{n}^{n-1}(R)$ is
central linear  Armendariz.
\end{thm}

\begin{proof} Suppose $R$ is central reduced ring. Let  $a^2=0$ for $a \in R$. Then $a \in C(R)$ and so $aRa=0$.
Since $R$ is semiprime, we have $a=0$. Therefore $R$ is reduced
and $T_{n}^{n-1}(R)$ is Armendariz by \cite[Theorem 5]{AC1}. Hence
$T_{n}^{n-1}(R)$  is linear Armendariz and  by Theorem \ref{thm2},
it is central linear Armendariz.  Conversely, assume that
$T_{n}^{n-1}(R)$ is central linear Armendariz. Using the similar
technique as in the proof of Lemma \ref{trr}, it can be shown that
$R$ is central reduced. \end{proof}

\end{document}